\documentclass{amsart}

\usepackage{amsmath,amsthm,amsfonts,latexsym,amssymb,mathrsfs,setspace,enumerate,color,cite,graphicx,epsf,mathtools,array,amstext,verbatim,float,fullpage}

\newtheorem{thm}{Theorem}
\newtheorem{lem}[thm]{Lemma}

\newtheorem{cor}[thm]{Corollary}
\numberwithin{equation}{section}
\numberwithin{thm}{section}

\newcommand{\intg}{\mathbb Z}

\newcolumntype{P}[1]{>{\centering\arraybackslash}p{#1}}
\newcommand{\legendre}[2]{\genfrac{(}{)}{}{}{#1}{#2}}
\newcommand{\sgn}{\mathrm{sgn}}

\title[Rotation Symmetries of Sequential Matrices]{Rotation Symmetries of Sequential Matrices with Applications to the Jacobi Symbol}
\author{Yemeen Ayub \and Charles L. Samuels}

\address{George Mason University, Department of Mathematical Sciences, 4400 University Drive, Fairfax, VA 22030}
\email{yayub@masonlive.gmu.edu}

\address{Christopher Newport University, Department of Mathematics, 1 Avenue of the Arts, Newport News, VA 23606}
\email{charles.samuels@cnu.edu}

\subjclass[2010]{11A05, 11A15 (Primary); 20B30 (Secondary)}
\keywords{Sequential Matrices, Completely Multiplicative Functions, Quadratic Reciprocity, Legendre Symbol, Jacobi Symbol}

\begin{document}

\begin{abstract}
	Suppose that $p$ is an odd prime and $\legendre{\cdot}{p}$ denotes the Legendre symbol modulo $p$.  If $p$ is has the form $p= n^2+1$ then one easily verifies that $\legendre{a}{p} = \legendre{-a}{p}$ for all 
	$a\in \intg/p\intg$. We identify various symmetry properties of sequential matrices over $\intg/(n^2+1)\intg$ regardless of whether $n^2+1$ is prime.   We deduce from these results a collection of symmetries involving
	Jacobi symbol modulo $n^2+1$ which generalize our above observation on the Legendre symbol.
\end{abstract}

\maketitle

\section{Introduction}

Suppose that $n$ is a positive integer and set $m = n^2 + 1$. An $n\times n$ matrix $A = (a_{i,j})$ with entries in $\intg/m\intg$ is called the {\it $n\times n$ sequential matrix} if
\begin{equation*}
	a_{i,j} = j + (i-1)n 
\end{equation*}
for all $1\leq i,j\leq n$.  The definition of sequential matrices may be restated recursively.  Indeed, we set $a_{1,1} = 1$, and for all other pairs $(i,j)$ we define
\begin{equation*}
	a_{i,j} = \begin{cases} a_{i,j-1} + 1 & \mbox{if } j > 1 \\ a_{i-1,n} + 1 & \mbox{if } j = 1.\end{cases}
\end{equation*}
This creates a matrix having a $1$ as the top left entry, and then counts consecutive integers moving rightward.  When we reach the end of a row, we move the next row and continue counting.
Some examples of sequential matrices are as follows:
\begin{equation*}
	\begin{pmatrix} 1 & 2 \\ 3& 4 \end{pmatrix},\quad \begin{pmatrix} 1 & 2 & 3 \\ 4& 5 & 6 \\ 7 & 8 & 9 \end{pmatrix}, 
		\quad \begin{pmatrix} 1 & 2 & 3 & 4 \\ 5& 6 & 7 & 8 \\ 9 & 10 & 11 & 12 \\ 13 & 14 & 15 & 16\end{pmatrix}.
\end{equation*}
In general, we will write $Q_n$ to denote the $n\times n$ sequential matrix.  Now suppose that $\varphi: \intg/m\intg\to \{0,1,-1\}$ is a completely multiplicative function and define 
$\varphi(A) = (\varphi(a_{i,j}))$.  The goal of this article is to study the symmetry properties of $\varphi(Q_n)$.

In order to make our agenda more precise, it is convenient to briefly describe the action of the dihedral group $D_4$ on the set of all $n\times n$ matrices with entries belonging to a given set $\Gamma$.   
For an arbitrary matrix $A = (a_{i,j})$, where $a_{i,j}\in \Gamma$, we define 
\begin{equation} \label{Generators}
	\tau(A) = (a_{j,i})\quad\mbox{and}\quad \rho(A) = (a_{j,n-i+1}).
\end{equation}
Clearly $\tau(A) = A^T$ so that $\tau$ flips $A$ across its main diagonal.  Although it is less obvious from the definition, a brief examination reveals that $\rho$ rotates $A$ clockwise by $90^\circ$.  

By applying $\tau$ and $\rho$ to the sequential matrix $Q_n$, we obtain generators of a subgroup of the symmetric group $S_{m}$ which is isomorphic to the dihedral group $D_4$ (Here, we are still assuming that
$m = n^2+1$).  
For the purposes of our discussion, we shall simply write $D_4(n)$ to denote this subgroup.  The definitions \eqref{Generators} ensure that $D_4(n)$ acts on the set of $n\times n$ matrices with entries in $\Gamma$.   
If $\Gamma$ is a commutative ring with unity and $J$ denotes the matrix having $1$'s along the off diagonal and $0$'s elsewhere, then all elements of $D_4(n)$ are determined according to the following formulas:

\begin{table}[H]
	\def\arraystretch{1.3}
	\begin{tabular}{c|c|c}
		Element $\sigma\in D_4(n)$ & Value of $\sigma(A)$ & Description of $\sigma(A)$ \\ \hline \hline
		$1$ & $A$ & Identity Map \\ \hline
		$\rho$ & $A^TJ$ & $90^\circ$ clockwise rotation \\ \hline
		$\rho^2$ & $JAJ$ & $180^\circ$ clockwise rotation \\ \hline
		$\rho^3$ & $JA^T$& $270^\circ$ clockwise rotation \\ \hline
		$\tau$ & $A^T$ & Flip across the main diagonal \\ \hline
		$\tau\rho$ & $JA$ & Flip across the horizontal center line \\ \hline
		$\tau\rho^2$ & $JA^TJ$ & Flip across the off diagonal \\ \hline
		$\tau\rho^3$ & $AJ$ & Flip across the vertical center line
	\end{tabular}
\end{table}

A matrix $A$ which is fixed by some element of $D_4(n)$ is often considered to have a symmetry property.  For example, in classical linear algebra, students are taught that $A$ is called {\it symmetric}
if $\tau(A) = A$.  Additionally, $A$ is called {\it centro-symmetric} if $\rho^2(A) = A$ and {\it Hankel-symmetric} if $\tau\rho^2(A) = A$.  Such symmetries have been studied extensively in various contexts
(see \cite{BrualdiMa, YSK, AbuJeib,AbuJeib2}, for example).  In the case of an $n\times n$ sequential matrix, our first result asserts that the rotation maps can be described by some simple arithmetic in $\intg/m\intg$.

\begin{thm} \label{RotoSymmetry}
	If $n$ is a positive integer and $m = n^2+1$ then $\rho(Q_n) = nQ_n$.
\end{thm}

Certainly we have that $n^2 \equiv -1 \mod m$.  In addition, multiplication by $n$ commutes with all elements of $D_4(n)$, and hence, we obtain the following table of values.

\begin{table}[H]
	\def\arraystretch{1.3}
	\begin{tabular}{c|c}
		Element $\sigma\in D_4(n)$ & Value of $\sigma(Q_n)$ in terms of $n$ \\ \hline\hline
		$1$ & $Q_n$ \\ \hline
		$\rho$ & $nQ_n$  \\ \hline
		$\rho^2$ & $-Q_n$  \\ \hline
		$\rho^3$ & $-nQ_n$ \\ \hline
		$\tau$ & $Q_n^T$ \\ \hline
		$\tau\rho$ & $nQ_n^T$  \\ \hline
		$\tau\rho^2$ & $-Q_n^T$  \\ \hline
		$\tau\rho^3$ & $-nQ_n^T$ 
	\end{tabular}
\end{table}

As we noted above, our plan is to study the symmetry properties of $\varphi(Q_n)$, where $\varphi: \intg/m\intg\to \{0,1,-1\}$ is a completely multiplicative function.  Thanks to Theorem \ref{RotoSymmetry},
we obtain a corollary which addresses this issue in relation to the rotation maps.

\begin{cor} \label{MultiplicativeRotoSymmetry}
	Suppose that $n$ is a positive integer and $m = n^2+1$.  If $\varphi: \intg/m\intg\to \{0,1,-1\}$ is a completely multiplicative function then $\varphi(\rho(Q_n)) = \varphi(n)\varphi(Q_n)$.
\end{cor}

It is worth noting that the expression $\sigma(\varphi(A))$ makes sense for all $\sigma \in D_4(n)$ and all $A\in M_{n\times n}(\intg/m\intg)$, and moreover, $\sigma(\varphi(A)) = \varphi(\sigma(A))$.  As a result, we may interpret 
Corollary \ref{MultiplicativeRotoSymmetry} as a statement about symmetry properties of $\varphi(Q_n)$.  These observations also enable us to deduce information about $\varphi(\sigma(Q_n))$ for other points $\sigma \in D_4(n)$.
For instance, we find that
\begin{equation*}
	\varphi(\rho^2(Q_n)) = \varphi(-1)\varphi(Q_n)\quad\mbox{and}\quad \varphi(\rho^3(Q_n)) = \varphi(-1)\varphi(n)\varphi(Q_n).
\end{equation*}
These assertions could also be concluded directly from Theorem \ref{RotoSymmetry} in conjunction with our above table of values.
	
\section{Applications}

If $p$ is an odd prime and $a\in \intg$ then the {\it Legendre symbol} $\legendre{a}{p}$ is defined so that
\begin{equation*}
	\legendre{a}{p} = \begin{cases} 0 & \mbox{if } a \equiv 0 \mod p \\
							1 &  \mbox{if } a \mbox{ is a perfect square in } (\intg/p\intg)^\times \\
							-1 &  \mbox{if } a \mbox{ is not a perfect square in } (\intg/p\intg)^\times. 
				\end{cases}
\end{equation*}
Certainly $a\mapsto \legendre{a}{p}$ is well-defined on $\intg/p\intg$, and moreover, this map defines a completely multiplicative function.  That is, for all $a,b\in \intg/p\intg$ we have that
\begin{equation*}
	\legendre{ab}{p} = \legendre{a}{p}\legendre{b}{p}.
\end{equation*}
For each odd integer $m>2$, write
\begin{equation*}
	m = \prod_{k=1}^K p_k
\end{equation*}
for its factorization into (not necessarily distinct) primes.  If $a\in \intg$ we define the {\it Jacobi symbol} $\legendre{a}{m}$ by
\begin{equation*}
	\legendre{a}{m} = \prod_{k=1}^K \legendre{a}{p_k}.
\end{equation*}
Like the Legendre symbol that it generalizes, the Jacobi symbol is a completely multiplicative function which is well-defined on $\intg/m\intg$.  For the purposes of applying Corollary \ref{MultiplicativeRotoSymmetry},
we are particularly interested in the case where $m = n^2+1$ for some positive even integer $n$.  If $A$ is a matrix with entries in $\intg/m\intg$ then we shall write $\legendre{A}{m}$ 
for the matrix obtained by applying the Jacobi symbol to each entry.

\begin{thm} \label{Jacobi}
	If $n$ is a positive even integer and $m=n^2+1$ then
	\begin{equation} \label{JacobiEquation}
		\legendre{\rho(Q_n)}{m} = \begin{cases} \legendre{Q_n}{m} & \mbox{if } n\equiv 0\mod 4 \\ -\legendre{Q_n}{m} & \mbox{if } n\equiv 2\mod 4. \end{cases}
	\end{equation}
\end{thm}

Theorem \ref{Jacobi} generalizes a well-known fact about the symmetry properties of the Jacobi symbol.  Under the assumptions of Theorem \ref{Jacobi}, $-1$ is certainly a square modulo $m$, and therefore, $-1$
is a square modulo every prime which divides $m$.   We conclude that $\legendre{-1}{m} = 1$ and
\begin{equation} \label{BasicSymmetry}
	\legendre{a}{m} = \legendre{-a}{m}\quad\mbox{for all }a\in \intg/m\intg.
\end{equation}
By applying Theorem \ref{Jacobi} twice, we can deduce \eqref{BasicSymmetry} in another way.  Indeed, it follows from Theorem \ref{Jacobi} that 
\begin{equation*}
	\legendre{\rho^2(Q_n)}{m} = \legendre{A}{m},
\end{equation*}
which is equivalent to the assertion that $\legendre{Q_n}{m}$ is centro-symmetric.  Now we immediately obtain \eqref{BasicSymmetry}.  In view of these observations, we may interpret Theorem \ref{Jacobi}
as an improvement to the well-known fact \eqref{BasicSymmetry}.

As a basic example of Theorem \ref{Jacobi}, consider the case where $n=4$ so that $m=17$ and we are in the situation of the 
first line of \eqref{JacobiEquation}.  After applying the Jacobi symbol to the $4\times 4$ sequential matrix, as predicted by Theorem \ref{Jacobi}, we obtain a matrix which is fixed under all rotations.

\begin{equation*}
	\legendre{Q_4}{17} = \begin{pmatrix}
		+1 & +1 & -1 & +1\\
		-1 & -1 & -1 & +1\\
		+1 & -1 & -1 & -1\\
		+1 & -1 & +1 & +1
	\end{pmatrix}
\end{equation*}

Of course, Theorem \ref{Jacobi} also applies in cases where $m$ is not prime or where $n\equiv 2\mod 4$.  For $n=6$ and $n=8$ we obtain the following matrices, respectively.

\begin{equation*}
	\legendre{Q_6}{37} =
		\begin{pmatrix}
 		+1 & -1 & +1 & +1 & -1 & -1 \\
 		+1 & -1 & +1 & +1 & +1 & +1 \\
 		-1 & -1 & -1 & +1 & -1 & -1 \\
 		-1 & -1 & +1 & -1 & -1 & -1 \\
 		+1 & +1 & +1 & +1 & -1 & +1 \\
 		-1 & -1 & +1 & +1 & -1 & +1 \\
		\end{pmatrix}
\end{equation*}

\begin{equation*}
	\legendre{Q_8}{65} =
		\begin{pmatrix}
		 +1 & +1 & -1 & +1 & 0 & -1 & +1 & +1 \\
		 +1 & 0 & -1 & -1 & 0 & +1 & 0 & +1 \\
 		-1 & +1 & -1 & 0 & -1 & -1 & -1 & -1 \\
 		0 & 0 & -1 & +1 & +1 & 0 & -1 & +1 \\
 		+1 & -1 & 0 & +1 & +1 & -1 & 0 & 0 \\
 		-1 & -1 & -1 & -1 & 0 & -1 & +1 & -1 \\
 		+1 & 0 & +1 & 0 & -1 & -1 & 0 & +1 \\
 		+1 & +1 & -1 & 0 & +1 & -1 & +1 & +1 \\
		\end{pmatrix}
\end{equation*}

In the case of $Q_6$, we observe that the matrix is not preserved under a $90^\circ$ clockwise rotation, but rather the sign is flipped after performing a $90^\circ$ clockwise rotation.
This is exactly as predicted by Theorem \ref{Jacobi}.
In the case of $Q_8$, we note that some entries are equal to $0$ due to the fact that the corresponding elements of $\intg/65\intg$ are not relatively prime to $65$.

\section{Proofs}

We begin with our proof of Theorem \ref{RotoSymmetry} concerning the behavior of rotations of sequential matrices.

\begin{proof}[Proof of Theorem \ref{RotoSymmetry}]
	Suppose that $Q_n = (a_{i,j})$ so we know that $\rho(Q_n) = (a_{j,n-i+1})$.  However, we have assumed that $Q_n$ is the $n\times n$ sequential matrix which means that $a_{i,j} = j + (i-1)n$.  Now we conclude that
	\begin{equation*}
		a_{j,n-i+1} = n-i+1 + (j-1)n = -i+1 + jn = jn + (i-1)(-1).
	\end{equation*}
	We also know that $n^2 \equiv -1 \mod m$, so in $\intg/m\intg$ we conclude that
	\begin{equation*}
		a_{j,n-i+1} = jn + (i-1)n^2 = n(j + (i-1)n) = na_{i,j}
	\end{equation*}
	completing the proof.
\end{proof}

The proof of Theorem \ref{Jacobi} is based on an application of Zolotarev's Lemma.

\begin{lem}\label{JacobiLemma}
	If $n$ is a positive even integer and $m = n^2+1$ then
	\begin{equation*}
		\legendre{n}{m} = (-1)^{n^2/4}. 
	\end{equation*}
\end{lem}
\begin{proof}
	Define the map $f:\intg/m\intg\to \intg/m\intg$ by $f(a) = na$, and since $n$ is relatively prime to $m$, we note that $f\in S_m$.  Then according to Zolotarev's Lemma \cite{Zolotarev},
	we have
	\begin{equation} \label{Zolotarev}
		\legendre{n}{m} = \sgn(f),
	\end{equation}
	where $\sgn(f)$ denotes the signature of $f$, i.e., $\sgn(f) = 1$ if $f$ is an even permutation and $\sgn(f) = -1$ if $f$ is an odd permutation.
	
	By Theorem \ref{RotoSymmetry}, applying $f$ to each entry of $Q_n$ performs a rotation of $Q_n$ clockwise by $90^\circ$.
	Therefore, every non-zero element $a\in \intg/m\intg$ satisfies $f^4(a) = a$, and since $n$ is even, $f^k(a)\ne a$ for all $0 < k < 4$.  In view of these observations, when we represent $f$ as a product of disjoint
	cycles, $f$ is a product of $n^2/4$ cycles of length four and $(0)$.  We know that $\sgn((0)) = 1$, and for every cycle $g$ of length four, we have that $\sgn(g) = -1$.  It now follows that
	$\sgn(f) = (-1)^{n^2/4}$ and we conclude the lemma from \eqref{Zolotarev}.
\end{proof}

We now obtain Theorem \ref{Jacobi} by simply applying Corollary \ref{MultiplicativeRotoSymmetry} with the Jacobi symbol in place of $\varphi$.

\begin{proof}[Proof of Theorem \ref{Jacobi}]
	By applying Corollary \ref{MultiplicativeRotoSymmetry} with the Jacobi symbol in place of $\varphi$ and utilizing Lemma \ref{JacobiLemma}, we immediately obtain that
	\begin{equation} \label{PreliminaryRotation}
		\legendre{\rho(Q_n)}{m} = \legendre{n}{m}\legendre{Q_n}{m} = (-1)^{n^2/4} \legendre{Q_n}{m}.
	\end{equation}
	If $n\equiv 0 \mod 4$ then $n^2/4$ is even and $(-1)^{n^2/4} =1$.  Otherwise, $n\equiv 2\mod 4$ and $n$ has the form $n = 4k+2$ for some $k\in \intg$.  This observation yields
	\begin{equation*}
		\frac{n^2}{4} = \frac{(4k+2)^2}{4} = \frac{16k^2 + 16k + 4}{4} = 4k^2 + 4k + 1
	\end{equation*}
	which is certainly odd meaning that $(-1)^{n^2/4} =-1$.  The result now follows from \eqref{PreliminaryRotation}.
\end{proof}

\bibliographystyle{abbrv}
\bibliography{LegendreMatrices-8-14-18}

\end{document}